\documentclass{article}
\usepackage[utf8]{inputenc}

\usepackage{longtable}
\usepackage{latexsym,amsmath,amstext,amsfonts,amscd,amsthm,upref,graphics}
\usepackage{enumitem,lipsum}
\usepackage{graphicx}
\usepackage[dvipsnames]{color,xcolor}
\usepackage[]{algorithm2e}
\usepackage{float}
\usepackage{url}
\restylefloat{table}
\usepackage[capitalise, noabbrev]{cleveref}

\usepackage{subcaption}

\newcommand{\ceiling}[1]{ \left\lceil #1  \right\rceil}

\usepackage[utf8]{inputenc}

\newcommand{\floor}[1]{\lfloor #1 \rfloor}

\newcommand{\sccd }{{SCCD}}
\newcommand{\lsccd }{{linear~SCCD}}

\newcommand{\scccd}{circular~SCCD}
\newcommand{\tscccd}{tight circular~SCCD}

\newcommand{\dccd }{{DCCD}}
\newcommand{\ldccd }{{linear~DCCD}}

\newcommand{\dcccd}{{circular DCCD}}
\newcommand{\tdcccd}{{tight circular~DCCD}}

\newcommand{\tldcccd }{{tight~linear~(circular)~DCCD}}
\newcommand{\ldcccd }{{linear~(circular)~DCCD}}

\newcommand{\cE}{\mathcal{E}}
\newcommand{\cL}{\mathcal{L}}

\newcommand{\radForExample}{2.7}
\newcommand{\radForExampleThree}{2.8}

\definecolor{colour1}{HTML}{009901}
\definecolor{colour2}{HTML}{FE0000}
\definecolor{colour3}{HTML}{3531FF}
\definecolor{colour4}{HTML}{E64EE5}
\definecolor{colour5}{HTML}{F56B00}
\definecolor{colour6}{HTML}{6200C9}
\definecolor{colour7}{HTML}{AAAD3A}

\definecolor{P2colour1}{HTML}{5ba300}
\definecolor{P2colour2}{HTML}{89ce00}
\definecolor{P2colour3}{HTML}{0073E6}
\definecolor{P2colour4}{HTML}{E6308A}
\definecolor{P2colour5}{HTML}{B51963}

\newcommand{\aDarkGreen}{P2colour1}
\newcommand{\aLightGreen}{P2colour2}

\newcommand{\aPink}{P2colour4}
\newcommand{\aBurgundy}{P2colour5}

\definecolor{alightgray}{HTML}{D3D3D3}

\title{Recursive and Cyclic Constructions for Double-Change Covering Designs}
\author{Amanda Lynn Chafee\footnote{Supported by Mathematics of Information Technology and Complex Systems (MITACS) research training award.},  Brett Stevens\footnote{Supported by the Natural Sciences and Engineering Research Council of Canada (funding reference number RGPIN 06392).}\\
	School of Mathematics and Statistics\\
	Carleton University\\
	1125 Colonel By Drive\\
	Ottawa, ON, K1S 5B6\\
	Canada\\
	\texttt{AmandaChafee@cmail.carleton.ca}, \texttt{brett@math.carleton.ca}
}

\date{November 2025}

\usepackage{tikz}
\usetikzlibrary{graphs,graphs.standard,backgrounds}

\usepackage[utf8]{inputenc}

\theoremstyle{plain}
\newtheorem{theorem}{Theorem}[{}]

\newtheorem{proposition}[theorem]{Proposition}
\newtheorem{corollary}[theorem]{Corollary}

\newtheorem{example}[theorem]{Example}

\usepackage{setspace}
\begin{document} 
	
	\maketitle
	
	\begin{abstract}
		
		
		A \textbf{double-change covering design} (\dccd{}) is a $v$-set $V$ and an ordered list $\cL$ of $b$ blocks of size $k$  where every pair from $V$ must occur in at least one block and each pair of consecutive blocks differs by exactly two elements. It is \textbf{minimal} if it has the fewest blocks possible and \textbf{circular} when the first and last blocks also differ by two elements. We give a recursive construction that uses 1-factorizations of complete graphs and expansion sets to construct a \dccd{}($v+\frac{v+k-2}{k-2},k,b+\frac{v}{k-2}\frac{v+k-2}{2k-4}$) from a \dccd{}($v,k,b$). We construct \dcccd{}($2k-2,k,k-1$) and \dcccd{}($2k-3,k,k-2$) from single-change covering designs and determine minimum DCCD when $v=2k-2$. We use difference like methods to construct five infinite families of minimum \dcccd{}($c(4k-6)+1,k,c^2(4k-6)+c$) when $1 \leq c\leq 5$ for any $k\geq 3$.  The recursive construction is then used to build twelve additional minimum \dccd{} from members of these infinite families.  Finally, the difference like method is used to construct a minimum \dcccd{}(61,4,366).
	\end{abstract}
	
	
	\section{Introduction}
	
	A \textbf{single-change covering design} (\sccd{}($v,k,b$)) $(V,\mathcal{L})$ is a $v$-set $V$ and an ordered list $\mathcal{L}= (B_1,B_2,...,B_b)$ of blocks of size $k$  where every pair, $S$, must occur in at least one block and each pair of consecutive blocks differs by exactly one element. If the first and last blocks also differ by one element and $\mathcal{L}$ is cyclically ordered, the \sccd{}  is called \textbf{circular} and otherwise is \textbf{linear}.  If we omit the qualifier, we are not specifying a particular type. If no \sccd{}($v,k,b')$ exists for all $b' < b$ then a \sccd{}($v,k,b$) is called \textbf{{minimum}}. The definitions for \textbf{double-change covering designs} are identical except that pairs of consecutive blocks differ by two elements. It is useful to modify the definition of covered to this specific context; a pair $S$ is \textbf{covered} in block $B_i$ if $S \subset B_i$ and $S \not\subset B_{i-1}$. All pairs of elements in the first block of a \lsccd{} (\ldccd{}) are covered. We say that element $x$  is \textbf{introduced} in block $B_i$ if $x\in B_i$ and $x\not\in B_{i-1}$. All elements are introduced in the first block of a \lsccd{} (\ldccd{)}. Throughout the paper, we use a superscript asterisk, $^*$, to emphasize the elements that are introduced in a block. For example, in Table~\ref{Table: tdcccd(10,4,2,9)}, the pair $\{0,1\}$ is covered in $B_7$ because both 0 and 1 are in $B_7$ and at least one is not in $B_6$. The pair $\{0,1\}$ is not covered in $B_8, B_9$ nor $B_1$ because both elements of the pair were in the preceding block. We say an element $x$ \textbf{leaves} in block $B_i$ if $x\in B_i$ and $x \not\in B_{i+1}$. All elements leave in the last block of a \ldccd{} (\lsccd{}). We denote an element leaving with a subscripted asterisk, $_*$.
	
	The earliest discussion of \sccd{}s was in 1969 and focused on applications to efficient computing by optimizing the movement of data between tape drives and core memory \cite{Nelder}. Retrieving new data from a tape drive was slower than accessing data in core memory so minimizing the number of times the tape was accessed vastly improved the run times of algorithms.  Thus the most efficient use of data already in core memory was essential.  It remains relevant today in cache efficient programming because accessing data stored in RAM is orders of magnitude slower than data stored in a processor cache \cite{drepper_what_2007}. 
	
	Later Wallis et al. proposed using \sccd{}s to balance the costs of changing components between reliability tests of a system being examined. For example, suppose we had nine components to compare three at a time and that running a test costs \$1 while changing any of the components between tests costs \$5 per component. An optimal covering design with 12 blocks, ($\{1,2,3\},\{4,5,6\},\{7,8,9\},\{1,4,7\},$ $\{2,5,8\}, \{3,6,9\}, \{1,5,9\},\{2,6,7\},$ $\{3,4,8\},$ $\{1,6,8\},\{2,4,9\},\{3,5,7\}$) would cost \$177 to run in the order shown, which is the worst case. A minimum \sccd{}(9,3,18) would cost \$118 even though it has more blocks. A minimum \dccd{}(9,3,12) would cost \$137. This cost efficiency motivated the study of \sccd{}. If instead, it costs \$10 to run a test and \$1 to change components these costs are \$153, \$200, and \$145 respectively. In the \$10 test, \$1/change scenario, in terms of cost, we see that the minimum \dccd{} out preformed the minimum \sccd{} and the covering where the blocks are given in random order. In order to optimize resources, it is vital to utilize the most efficient designs for your scenario.
	
In \cref{sec_background} we briefly survey recent history on the subject as well as prerequisite knowledge required. We construct linear and circular double-change covering designs by using expansion sets and 1-factorizations of complete graphs in \cref{recursion_sec}.  In \cref{difference_sec} we provide constructions using difference like methods. 
	
	\section{Background}\label{sec_background}
	\subsection{Recent History}

	Subsets of Constable, Gower, McSorley, Phillips, Preece, Van Rees, Wallis, Yucas, and Zhang substantially developed the theory of \sccd{}s in the 1990s \cite{Gower, circular, Phillips_20, Phillips_12, tsccd, sccd, Zhang_new_bounds} and our paper \cite{Chafee_sccd} of 2023 also advanced the subject. Briefly the current state of knowledge is that we have constructions for all minimum single-change covering designs of block sizes two, three and four, as well as four classes for $v \pmod{16}$ with block size five. We have fixed blocksize recursive constructions for \sccd{} and difference method like constructions for circular \sccd{} for arbitrary $k$. For a more complete history please see \cite{Chafee_sccd} and \cite{inside}. 
	
	As seen in the example above, modifying the number of elements changed between consecutive blocks can be used to optimize costs when $m$-change designs are used in testing. We generally only consider $v\geq k+m$ if $b>1$ in a $m$-change design because we need at least $m$ additional elements to change between blocks. Preece et al.\cite{tsccd} introduce the concept of double-change covering designs ($m=2$) and produce an example of a minimum \dccd{}(9,3,12) shown in Table~\ref{tight dccd(9,3,12)}. Lindbloom and McSorley begin a more in-depth exploration of double-change covering designs in Lindbloom's M.Sc thesis \cite{lindbloom_tight_2011}.  Lindbloom gives a construction for \ldccd{}$(5,2,10)$ and two general constructions for odd $v$ and even $v$ linear and circular \dccd{}$(v,2,b)$ when $v\geq 6$. She concludes by noting that if you have a double-change covering design with $k=2$, then a Hamilton cycle exists in $K_v$ for all $v\geq 6$ as a consequence of Dirac's Theorem.
	
	Gamachchige and McSorley continue the exploration in Gamachchige's M.Sc and Ph.D. theses.  In his M.Sc \cite{gangoda_gamachchige_double-change_2014}, Gamachchige gives a construction for minimum or near minimum  \dcccd{}($v,3,b$) for all $v$.  When $v \equiv 0,3,5 \pmod {6}$, he uses the Bose construction from idempotent commutative quasigroups and methodically reorders the blocks.  For $v \equiv 1 \pmod{6}$ he modifies a construction of Skolem. All the preceding are minimum.  To build a \dcccd{}($v,3,b$) when $v \equiv 2,4 \pmod{6}$ Gamachchige modifies the construction for $v-1$ noting that the results are not necessarily minimum.  
	
	\begin{table}[ht!]
		\begin{center}
			\begin{tabular}{llllllllllll}
				1$^*$ & 1     & 1     & 1     & 4$^*$ & 3$^*$ & 2$^*$ & 3$^*$ & 8$^*$ & 8     & 8     & 2$^*$ \\
				2$^*$ & 4$^*$ & 5$^*$ & 6$^*$ & 6     & 6     & 6     & 5$^*$ & 9$^*$ & 2$^*$ & 3$^*$ & 9$^*$ \\
				3$^*$ & 7$^*$ & 9$^*$ & 8$^*$ & 5$^*$ & 9$^*$ & 7$^*$ & 7     & 7     & 5$^*$ & 4$^*$ & 4    
			\end{tabular}
		\end{center}
		\caption{A minimum \dccd{}(9,3,12)}
		\label{tight dccd(9,3,12)}
	\end{table}

	In his Ph.D. \cite{gangoda_gamachchige_double-change_2017},  Gamachchige examines the regularity of elements in a \dccd{} and finds that linear \dccd{}($3k-5,k,2k-4)$ are minimum. Additionally, he proved that if $2k-3$ is an odd prime power then minimum \ldccd{}($v,k,(2v-2)/3)$ can only exist if $v \geq 3k-5$. He went on to show that fixed elements, an element that appears in every block, do not exist in minimum linear \dccd{}($v,k,\ceiling{\frac{v(v-1)-k(k-1)}{4(k-2)+2}}+1)$ for $v\geq 3k-6$ nor minimum \dccd{}($v,4,\ceiling{(v^2-v-12)/10}$) when $v \geq 6$. However, minimum \dcccd{}($v,k,\ceiling{\frac{v(v-1)}{4(k-2)+2}}$) have a fixed element if and only if $v=2k-3$.  Gamachchige shows that minimum \dccd{}($v,2,v(v-1)/2$) exist for all $v \geq 6$. 
	He provides constructions for minimum \dccd{}($v,4,b$) when $v=6,...,14$ as well as recursive constructions that preserve minimality. The thesis concludes with some general constructions for \dcccd{}($v,4,b$) that are not necessarily minimum.

	\subsection{Preliminaries}
	
	In this section we gives some necessary background on \dccd{} and other combinatorial objects we will use.
	\begin{theorem}\label{bound_thm}
		There are at least $g_1(v,k)=\frac{\binom{v}{2}-\binom{k}{2}}{2k-3}+1$ blocks in a \ldccd{} and $g_2(v,k)=\frac{\binom{v}{2}}{2k-3}$ blocks in a \dcccd{}.
	\end{theorem}
	\begin{proof}
		
		The first block of a linear \dccd{} will cover $\binom{k}{2}$ pairs. Every other block will have two elements introduced. In these blocks, there is one pair covered that contains both introduced elements and $2(k-2)$ additional pairs covered containing exactly one of the introduced elements. There are $\binom{v}{2}$ pairs to cover, so $\binom{v}{2}\leq \binom{k}{2}+(b-1)\bigg(2(k-2)+1\bigg)$. In the circular case, each block covers at most $2(k-2)+1$ pairs and the calculation is similar. 
	\end{proof}

	A \ldcccd{}($v,k,b$) is \textbf{{economical}} if it has $\ceiling{g_1(v,k)}$  $\big(\ceiling{g_2(v,k)}\big)$ blocks and \textbf{{tight}} if it is economical and $g_1(v,k)$ $\big(g_2(v,k)\big)$ is an integer. By Theorem~\ref{bound_thm}, any economical \dccd{} is minimum. In a tight \dccd{} every pair is covered exactly once. In an economical \dccd{}, there are less than $2k-3$ additional pairs covered.  An example of a \tldcccd{}(7,3,7) is given in Table~\ref{Table: dccd(7,3,2)}. The blocks are those of a $STS(7)$. When $k=3$ in a \dccd{}, $g_1 = g_2$ and so the $\dccd{}$~in Table~\ref{Table: dccd(7,3,2)} can be considered either linear or circular when block order allows. 
	
	The set of elements $ B_i \cap B_{i+1}$ is the $i^{th}$ \textbf{unchanged subset}, $U_i$ for $1 \leq i < b$. In  a linear \dccd{}, as no block precedes $B_1$ nor follows $B_b$, every $(k-2)$-subset of $B_1$ and every $(k-2)$-subset of $B_b$ are unchanged subsets occurring before $B_1$ and after $B_b$ respectively.  Any of these can be designated $U_0$ and $U_b$ respectively. However, circular \dccd{} have the double-change property between the first and last block, so $U_0= B_1 \cap B_b =U_b$. For example, thinking of the \dccd{}(7,3,7) in Table~\ref{Table: dccd(7,3,2)} as linear, it has the unchanged subsets $U_1 = \{0\},\, U_2 = \{4\},\, U_3 = \{3\},\, U_4 = \{6\},\, U_5 = \{6\},\, U_6 = \{5\}$ where $U_0$ can be any of $\{0\},\,\{1\},$ or $\{2\}$ and $U_7$ can be any of $\{1\},\,\{3\},$ or $\{5\}$. Considering it now as circular, it has the same $U_1,...,U_6$; however, $U_0 = U_7 =\{1\}$.  Any collection of $\{U_{i_j}\}_{j=1}^{v/(k-2)}$ which is a partition of $V$ is called an \textbf{expansion set}. Note that to have an expansion set we necessarily require $(k-2) | v$. The circular \dccd{}~in Table~\ref{Table: dccd(7,3,2)} does not have an expansion  set because $2$ does not occur in any $U_i$. The freedom of choice for $U_0$ and $U_b$ in the linear \dccd{}(7,3,7) given in Table~\ref{Table: dccd(7,3,2)} gives us an expansion set where $U_0 = \{2\}$. Expansion sets will be used in Section~\ref{recursion_sec} to prove a recursive construction.  In this paper, we use carets, $\land$, to indicate the unchanged subsets between blocks that are in an expansion set. 
	
	\begin{table}[ht!]
		\centering
		\setlength{\tabcolsep}{-1pt}
		\begin{tabular}{clclclclclclclc} 
			& $B_1$ &   & $B_2$  &  & $B_3$ &   & $B_4$ &   & $B_5$  &  & $B_6$    &  & $B_7$    \\
			&0$^*$ & & 0     & & 2$^*$ && 0$^*$ && 1$^*$ && 5$^*$ && 5     \\
			&1 & &4$^*$ & & 4     &  & 6$^*$ &  & 6     &  & 6     &  &3$^*$ \\
			
			$_{\land}$ & 2$^*$  & $_{\land}$ &  5$^*$& $_{\land}$ & 3$^*$ &  $_{\land}$ &  3  &   & 4$^*$& $_{\land}$ & 2$^*$ & $_{\land}$ & 1$^*$ & $_{\land}$\\
			
			\tiny{$U_0 = \{2\}$} & & \tiny{$U_1 = \{0\}$} & & \tiny{$U_2 = \{4\}$} & &\tiny{$U_3 = \{3\}$} & & \phantom{\tiny{$U_4= \{6\}$}} & & \tiny{$U_5 = \{6\}$} & &\tiny{$U_6 = \{5\}$} & & \tiny{$U_7 = \{1\}$}
			
		\end{tabular}
		\setlength{\tabcolsep}{6pt}
		\caption{A \tldcccd{}(7,3,7)}
		\label{Table: dccd(7,3,2)}\label{Table: dccd(7,3,2,7)}
	\end{table}
	
	We will be using some graph theory in this work. A \textbf{{perfect matching} or {1-factor}} of a graph $G=(V,E)$ is a subset of $E$ of cardinality $|V|/2$ whose pairs of end-vertices partition $V$ when $|V|$ is even.  A \textbf{1-factorization} of $(V,E)$ is a decomposition of $E$ into 1-factors.  Each color in Figure~\ref{Figure: tikz K_6 1-factor} is a 1-factor in a 1-factorization of $K_6$. A 1-factorization of $K_n$ exists for all even $n$ \cite{Factors_and_Factorizations}. 
	
	\begin{figure}[ht]
		\begin{center}
			\begin{tikzpicture}
	
	\foreach \x [count=\idx from 0] in {a,...,f} {
		\pgfmathparse{90 + \idx * (360 /6)}
		\node at (\pgfmathresult:2.0cm) [circle,draw=black,fill=white, line width=0.75mm] (\idx) {\x};
	};

	\begin{pgfonlayer}{background}
		\draw [-,cap=rect,line width=1mm,draw=P2colour5] (1)--(4); 
		\draw [-,cap=rect,line width=1mm,draw=P2colour5] (2)--(0); 
		\draw [-,cap=rect,line width=1mm,draw=P2colour5] (3)--(5); 
		
		\draw [-,cap=rect,line width=1mm,draw=P2colour4] (2)--(4); 
		\draw [-,cap=rect,line width=1mm,draw=P2colour4] (1)--(3); 
		\draw [-,cap=rect,line width=1mm,draw=P2colour4] (5)--(0); 
		
		\draw [-,cap=rect,line width=1mm,draw=P2colour1] (3)--(4); 
		\draw [-,cap=rect,line width=1mm,draw=P2colour1] (2)--(5); 
		\draw [-,cap=rect,line width=1mm,draw=P2colour1] (1)--(0); 
		
		\draw [-,cap=rect,line width=1mm,draw=black] (5)--(4); 
		\draw [-,cap=rect,line width=1mm,draw=black] (1)--(2); 
		\draw [-,cap=rect,line width=1mm,draw=black] (3)--(0); 
		
		\draw [-,cap=rect,line width=1mm,draw=P2colour3] (0)--(4); 
		\draw [-,cap=rect,line width=1mm,draw=P2colour3] (1)--(5); 
		\draw [-,cap=rect,line width=1mm,draw=P2colour3] (2)--(3); 
	\end{pgfonlayer}
	
\end{tikzpicture}
		\end{center}
		\caption{1-factorization of $K_6$}
		\label{Figure: tikz K_6 1-factor}
	\end{figure}

	\section{Constructions}
	
	In this paper, we prove a recursive construction in Section~\ref{recursion_sec} and a direct construction using difference methods in Section~\ref{difference_sec}.  We will be able to apply the recursive construction to some examples constructed cyclically to build further double-change covering designs.
	
	\subsection{Expansion Set Recursion}\label{recursion_sec}
	
	Preece et al.\cite{tsccd} use expansion sets to construct a \sccd{}($v+1,k,b$) from a \sccd{}($v,k,b$). This has been extended to \scccd{} \cite{circular}. We now do something similar with double-change covering designs by utilizing 1-factorizations of a complete graph.
	
	\begin{theorem}\label{Theorem: dccd(v+1,k,2) increase}
		Let $v$ be an odd multiple of $k-2$ and let $(V,\cL)$ be a \text{\dccd{}}($v,k,b$) with an expansion set. Then a \dccd{}($v^*, k,b^*$), ($V^*,\cL^*$) exists for $v^*=v+\frac{v}{k-2}+1$ and $b^*= b+\frac{1}{2}\frac{v}{k-2}(\frac{v}{k-2}+1)$. Moreover, if ($V,\cL$) is tight, economical or circular then ($V^*,\cL^*$) is too, respectively. 
	\end{theorem}
	
	\begin{proof}
		Let ($V,\cL)$ be a \ldccd{}($v,k,b$) with the expansion set $\cE = \{U_{i_j} : 1 \leq j \leq l = \frac{v}{k-2}\}$. Let $\{F_j\}_1^l$ be a 1-factorization of $K_{l+1}$ on a set $X$ disjoint from $V$ \cite{Factors_and_Factorizations}. 
		
		For each $1 \leq j \leq l$ and each edge $e\in F_j$ let $B_{j,e}= U_{i_j}\cup e$. To create $\cL^*$ from $\cL$, insert all the $B_{j,e}$ in any order at each expansion set location ${i_j}$. These insertions maintain the double-change property. All the pairs that were covered in $(V,\cL)$ are still covered in $\cL^*$. Since each pair of elements from $X$ is in exactly one 1-factor all the pairs in $X$ are covered. The blocks inserted at expansion set locations ensure each pair of elements $\{x,y\}$, $x\in X$, $y\in V$ is covered. So $(V^*,\cL^*)$ is a \dccd{}($v^*,k,b^*$). 
		
		Each $B_{j,e}$ only covers new pairs and does not effect introductions in preexisting blocks.  Further, the pairs covered in the $B_{j,e}$ are only covered in that particular block so if ($V,\cL$) was tight or economical so is $(V^*,\cL^*)$. The block insertions do not affect the double-change between the first and last block, so if $(V,\cL)$ was circular so is ($V^*,\cL^*$). 
	\end{proof}
	
	For example consider the linear \dccd{}(7,3,7) and expansion set in Table~\ref{Table: dccd(7,3,2,7)}. We may insert four new blocks at each expansion set location that consist of the unchanged subset and the vertices of an edge in a 1-factor of $K_8$ in Figure~\ref{Figure: K8} to construct a \ldccd{}(15,3,35) as seen in Table~\ref{Table: DCCD(15,3,2,3)}.
	
	\begin{figure}[htb!]
		\centering
		\begin{tikzpicture}
			
			\foreach \x [count=\idx from 8] in {8,...,14} {
				\pgfmathparse{40 + \idx * (360 /7)}
				\node at (\pgfmathresult:2.4cm) [circle,draw=black,fill=white, line width=0.75mm] (\idx) {\x};
			};
			
			\node[circle,draw=black,fill=white,line width=0.75mm] (15) at (0:0) {15};
			
			\begin{pgfonlayer}{background}
				{
					\draw [-,cap=rect,line width=1mm,draw=black] (8)--(9);
					\draw [-,cap=rect,line width=1mm,draw=black] (10)--(14);
					\draw [-,cap=rect,line width=1mm,draw=black] (11)--(13);
					\draw [-,cap=rect,line width=1mm,draw=black] (12)--(15);
					
					\draw [-,cap=rect,line width=1mm,draw=black] (8)--(10);
					\draw [-,cap=rect,line width=1mm,draw=black] (11)--(14);
					\draw [-,cap=rect,line width=1mm,draw=black] (12)--(13);
					\draw [-,cap=rect,line width=1mm,draw=black] (9)--(15);
					
					\draw [-,cap=rect,line width=1mm,draw=black] (8)--(11);
					\draw [-,cap=rect,line width=1mm,draw=black] (9)--(10);
					\draw [-,cap=rect,line width=1mm,draw=black] (12)--(14);
					\draw [-,cap=rect,line width=1mm,draw=black] (13)--(15);
					
					\draw [-,cap=rect,line width=1mm,draw=black] (8)--(15);
					\draw [-,cap=rect,line width=1mm,draw=black] (9)--(14);
					\draw [-,cap=rect,line width=1mm,draw=black] (10)--(13);
					\draw [-,cap=rect,line width=1mm,draw=black] (11)--(12);
					
					\draw [-,cap=rect,line width=1mm,draw=black] (8)--(12);
					\draw [-,cap=rect,line width=1mm,draw=black] (13)--(14);
					\draw [-,cap=rect,line width=1mm,draw=black] (9)--(11);
					\draw [-,cap=rect,line width=1mm,draw=black] (10)--(15);
					
					\draw [-,cap=rect,line width=1mm,draw=black] (8)--(13);
					\draw [-,cap=rect,line width=1mm,draw=black] (9)--(12);
					\draw [-,cap=rect,line width=1mm,draw=black] (10)--(11);
					\draw [-,cap=rect,line width=1mm,draw=black] (14)--(15);
					
					\draw [-,cap=rect,line width=1mm,draw=black] (8)--(14);
					\draw [-,cap=rect,line width=1mm,draw=black] (9)--(13);
					\draw [-,cap=rect,line width=1mm,draw=black] (10)--(12);
					\draw [-,cap=rect,line width=1mm,draw=black] (11)--(15);
				}
				
				{
					\draw [-,cap=rect,line width=1mm,draw=colour1] (8)--(9);
					\draw [-,cap=rect,line width=1mm,draw=colour1] (10)--(14);
					\draw [-,cap=rect,line width=1mm,draw=colour1] (11)--(13);
					\draw [-,cap=rect,line width=1mm,draw=colour1] (12)--(15);
				}
				
				{
					\draw [-,cap=rect,line width=1mm,draw=colour2] (8)--(10);
					\draw [-,cap=rect,line width=1mm,draw=colour2] (11)--(14);
					\draw [-,cap=rect,line width=1mm,draw=colour2] (12)--(13);
					\draw [-,cap=rect,line width=1mm,draw=colour2] (9)--(15);
				}
				
				{
					\draw [-,cap=rect,line width=1mm,draw=colour3] (8)--(11);
					\draw [-,cap=rect,line width=1mm,draw=colour3] (9)--(10);
					\draw [-,cap=rect,line width=1mm,draw=colour3] (12)--(14);
					\draw [-,cap=rect,line width=1mm,draw=colour3] (13)--(15);
				}
				
				{
					\draw [-,cap=rect,line width=1mm,draw=colour4] (8)--(15);
					\draw [-,cap=rect,line width=1mm,draw=colour4] (9)--(14);
					\draw [-,cap=rect,line width=1mm,draw=colour4] (10)--(13);
					\draw [-,cap=rect,line width=1mm,draw=colour4] (11)--(12);
				}
				
				{
					\draw [-,cap=rect,line width=1mm,draw=colour5] (8)--(12);
					\draw [-,cap=rect,line width=1mm,draw=colour5] (13)--(14);
					\draw [-,cap=rect,line width=1mm,draw=colour5] (9)--(11);
					\draw [-,cap=rect,line width=1mm,draw=colour5] (10)--(15);
				}
				
				{
					\draw [-,cap=rect,line width=1mm,draw=colour6] (8)--(13);
					\draw [-,cap=rect,line width=1mm,draw=colour6] (9)--(12);
					\draw [-,cap=rect,line width=1mm,draw=colour6] (10)--(11);
					\draw [-,cap=rect,line width=1mm,draw=colour6] (14)--(15);
				}
				
				{
					\draw [-,cap=rect,line width=1mm,draw=colour7] (8)--(14);
					\draw [-,cap=rect,line width=1mm,draw=colour7] (9)--(13);
					\draw [-,cap=rect,line width=1mm,draw=colour7] (10)--(12);
					\draw [-,cap=rect,line width=1mm,draw=colour7] (11)--(15);
				};
			\end{pgfonlayer}

		\end{tikzpicture}
		\caption{K$_8$}
		\label{Figure: K8}
	\end{figure}
	
	\begin{table}[ht!]
		\centering
		\setlength{\tabcolsep}{.5\tabcolsep}
		\begin{tabular}{llllllllllllllllll}
			
			{$B_1$} & {$B_{2}$} & {$B_{3}$} & {$B_{4}$} & {$B_{5}$} & {$B_{6}$} & {$B_{7}$} & {$B_{8}$} & {$B_{9}$}  & {$B_{10}$}  & {$B_{11}$} & {$B_{12}$} \\
			\textcolor{colour1}{8} & \textcolor{colour1}{10} & \textcolor{colour1}{11} & \textcolor{colour1}{12} & 0$^*$ 
			& {0}  & {0}  & {0}  & {0} & 0 
			& \textcolor{colour3}{8}  & \textcolor{colour3}{9}  \\
			\textcolor{colour1}{9} & \textcolor{colour1}{14} & \textcolor{colour1}{13} & \textcolor{colour1}{15} & 1$^*$ 
			& \textcolor{colour2}{8}  & \textcolor{colour2}{11} & \textcolor{colour2}{12} & \textcolor{colour2}{9} & 4$^*$
			& {4}  & {4}  
			\\
			{2} & {2}  & {2}  & {2} & 2 
			& \textcolor{colour2}{10} & \textcolor{colour2}{14} & \textcolor{colour2}{13} & \textcolor{colour2}{15} & 5$^*$ 
			& \textcolor{colour3}{11} & \textcolor{colour3}{10} 
			\\ \hline 
			{$B_{13}$} & {$B_{14}$} & {$B_{15}$} & {$B_{16}$} & {$B_{17}$} & {$B_{18}$} &
			{$B_{19}$} & {$B_{20}$} & {$B_{21}$}  & {$B_{22}$} & {$B_{23}$} & {$B_{24}$}  \\
			 \textcolor{colour3}{12} & \textcolor{colour3}{13} & 2$^*$ 
			& \textcolor{colour4}{8}  & \textcolor{colour4}{9} & \textcolor{colour4}{10} & \textcolor{colour4}{11} & 0$^*$ & 1$^*$ 
			& \textcolor{colour5}{8}  & \textcolor{colour5}{13} & \textcolor{colour5}{9}   \\
			 {4} & {4}  & 4 &\textcolor{colour4}{15} & \textcolor{colour4}{14} & \textcolor{colour4}{13}  & \textcolor{colour4}{12} & 6$^*$ & 6 
			& {6}  & {6}  & {6}  \\
			\textcolor{colour3}{14} & \textcolor{colour3}{15} & 3$^*$ 
			& {3}  & {3} & {3} & {3}  & 3 & 4$^*$ 
			& \textcolor{colour5}{12} & \textcolor{colour5}{14} & \textcolor{colour5}{11}   \\ \hline
			{$B_{25}$} & {$B_{26}$} & {$B_{27}$} & {$B_{28}$} & {$B_{29}$} & {$B_{30}$} & {$B_{31}$} & {$B_{32}$} & {$B_{33}$} & {$B_{34}$} & {$B_{35}$} \\
			\textcolor{colour5}{10} & 5$^*$ & {5}  & {5}  & {5}  & {5} & 5 
			& \textcolor{colour7}{8}  & \textcolor{colour7}{9}  & \textcolor{colour7}{10} & \textcolor{colour7}{11} \\
			 {6} & 6 
			& \textcolor{colour6}{8}  & \textcolor{colour6}{9}  & \textcolor{colour6}{10} & \textcolor{colour6}{14} & 3$^*$ 
			& \textcolor{colour7}{14} & \textcolor{colour7}{13} & \textcolor{colour7}{12} & \textcolor{colour7}{15}\\
			\textcolor{colour5}{15} & 2$^*$ & \textcolor{colour6}{13} & \textcolor{colour6}{12} & \textcolor{colour6}{11} & \textcolor{colour6}{15} & 1$^*$ 
			& {1}  & {1}  & {1}  & {1} 
			
		\end{tabular}
		\caption{A \ldccd{}(15,3,35) built from the \ldccd{}(7,3,7) of Table 2 and $K_8$}
		\label{Table: DCCD(15,3,2,3)}
	\end{table}

	We note that every new element $x\in X$ in Theorem~\ref{Theorem: dccd(v+1,k,2) increase} is always removed one block after it is introduced. Consequently, element $x$ is never in an unchanged subset so $\cL^*$ never has an expansion set.
	
	\subsection{Cyclic Constructions}\label{difference_sec}
	
	We can construct \dcccd{}($v,k,b$) using difference methods. McSorley proved in \cite{circular} that $B_0 = \{0,1,...,k-1\}$ is the  base block of a \scccd{}$(2k-1,k,2k-1)$ over $\mathbb{Z}_{2k-1}$. This theorem is not stated in terms of difference like methods but is equivalent to a difference method construction. In \cite{Chafee_sccd} we use difference methods to construct \tscccd{}$(2c(k-1)+1,k,2c^2 (k - 1) + c )$ for all $k\geq 2, c\geq 1$, where $c$ is the number of base blocks we use to start the construction. Here we extend these methods to \dcccd{}.  In McSorley's construction for \tscccd{}$(2k'-1,k',2k'-1)$ the circular list of blocks is $\cL' = (B_0,B_0+1,\ldots, B_0 + i, \ldots, B_0+v-1)$ with addition in $\mathbb{Z}_{2k'-1}$. We apply a doubling construction to this family which is similar to Example 5.2.2. in \cite{gangoda_gamachchige_double-change_2017}.
	
	\begin{theorem}\label{Theorem: tdcccd(4k-2,2k,2,2k-1) exist}
		\label{Theorem: tdcccd(4k-2,2k,2k-1) exist}
		\label{Theorem: tdcccd(2k-2,k,2,k-1) exist, k=2k'}
		\label{Theorem: tdcccd(2k-2,k,k-1) exist, k=2k'}
		A \tdcccd{}$(2k-2,k,k-1)$ exists for all $k\geq 2$ if and only if $k$ is even.
	\end{theorem}
	
	\begin{proof}
		When $k = 2k'$ is even, in McSorley's construction, we replace $x\in \mathbb{Z}_{2k'-1}$ in each block with $2x$ and $2x+1$ in $\mathbb{Z}_{4k'-2}$ to produce a list of $2k'-1$ blocks $B_i$ of size $k=2k'$ from $\mathbb{Z}_{2k-2}$.  Because McSorley's \tscccd{}$(2k'-1,k',2k'-1)$ is single-change, the resulting design is double-change. The pairs $\{2x,2x+1\} \subset \mathbb{Z}_{2k-2}$ are covered when introduced. All other pairs $\{y,z\} \subset \mathbb{Z}_{2k-2}$  are covered in the block constructed from the block of the \scccd{} where $\{\floor{\frac{y}{2}},\floor{\frac{z}{2}}\}$ is covered. The number of blocks matches $g_2(2k-2,k)$ and thus this constructs a \tdcccd{}$(2k-2,k,k-1)$
		
		Now suppose that $k$ is odd and a \dcccd{}$(2k-2,k,k-1)$ exists.  Since $g_2(2k-2,k)=k-1$ is an integer this \dccd{} would be tight and every pair must be covered exactly once. Since we have two elements introduced each block, there are  $2(k-1)$ introductions, and thus the average number of introductions is one.  It follows then that there is an element introduced more than once if and only if there is an element not introduced at all; i.e. an element in every block. However, if $x$ is in every block and $y$ is introduced more than once, then the pair $\{x,y\}$ is covered more than once; this contradicts the design being tight. Thus we conclude that every element is introduced exactly once.
		
		We now calculate how many blocks any single element, $x$, is in. When $x$ is introduced it appears in $k-1$ covered pairs. In each block where $x$ remains, $x$ appears in two more covered pairs. In any \dccd{}, let $r_x$ be the number of blocks $x$ is in and $i_x$ be the number of times $x$ is introduced. Then $x$ appears in $i_x(k-1) + 2(r_x-i_x)$ covered pairs. In any \dccd, the fact that every pair containing $x$ is covered gives $i_x(k-3) + 2r_x = v-1 + e_x$ where $e_x$ is the number of excess covered pairs containing $x$.  In this case we know that $i_x=1$, every pair is covered exactly once ($e_x=0$) and $v = 2k-2$, yielding $r_x = k/2$. This is not an integer, but it represents a cardinality, so we derive a contradiction. Thus a \tdcccd{}($2k-2,k,k-1$) does not exist when $k$ is odd.
	\end{proof}
	
	For example, when $k$ is even, we construct a \tdcccd{}(6,4,3) from the \tscccd{}(3,2,3) in Table~\ref{Table: tdcccd(6,4,2,3)}.
	
	\begin{table}[thb!]
		\centering
		\begin{tabular}{llllllllll}
			$B'_1$ & $B'_2$ & $B'_3$   &&&&&$B_1$ & $B_2$ & $B_3$\\
			0 & 2$^*$ &2            &&&&& 0 & 4$^*$ & 4 \\
			1$^*$ & 1 & 0$^*$       &&&&&  1 & 5$^*$ & 5 \\
			&&&&&&&                        2$^*$ & 2 & 0$^*$ \\
			&&&&&&&                        3$^*_{\land}$ & 3$_{\land}$ & 1$^*_{\land}$
		\end{tabular}
		\caption{A \tscccd{}(3,2,3) and \tdcccd{}(6,4,3)}
		\label{tdcccd(6,4,2,3)}
		\label{Table: tdcccd(6,4,2,3)} \label{Table: tscccd(3,2,2,3)}
	\end{table}
	
	A \tdcccd{}($2k-2,k,k-1$) cannot exist when $k$ is odd.  We show by construction that the minimum \dcccd{}$(2k-2,k,b)$ has $b = k$ when $k$ is odd for all $k\geq 5$. 
	
	\begin{theorem} \label{Theorem: minimum construction for dcccd(2k-2,k,k)}
		A minimum \dcccd{}($2k-2,k,k$) exists for all odd $k\geq 5$.
	\end{theorem}
	\begin{proof}
		We consider only odd $k\geq 5$ because we require  $v \geq k+2$. Let $V = \{\infty_s,0_s,1_s, 2_s, \hdots, (k-3)_s,\infty_l, 0_l, 1_l,2_l, \hdots, (k-3)_l\}$.  Both $\infty_s$ and $\infty_l$ will be introduced twice and each time leave immediately after introduction. The remaining $2k-4$ elements are introduced exactly once and evenly distributed between ``short'' and ``long'' runs of consecutive blocks. Elements in the ``short'' runs are sub-scripted with $s$ and will be in exactly $\frac{k+1}{2}$ consecutive blocks while elements in the ``long'' runs are subscripted with $l$ and will be in exactly $\frac{k+1}{2}+1$ consecutive blocks.
		
		We label the blocks from $B_0$ to $B_{k-1}$ and index blocks modulo $k$. We highlight the three blocks containing either $\infty_s$ or $\infty_l$:
		\begin{align*}
			B_0 &= \{\infty_l, 0_l,0_s,1_l,2_s,3_l, ..., (k-3)_s\}, \\
			B_{\frac{k-1}{2}}&=\{\infty_l, \infty_s, 0_l,1_s, 2_l,3_s,\hdots,(k-3)_l\}, \\
			B_{k-1} &= \{\infty_s, 0_s,1_l,2_s,3_l,4_s,\hdots, (k-4)_l, (k-3)_s, (k-3)_l\}. 
		\end{align*}
		
		 Element $\infty_s$ is introduced in blocks $\frac{k-1}{2}$ and $k-1$. Element $\infty_l$ is introduced in blocks $0$ and $\frac{k-1}{2}$. Element $i_s\in V$ is introduced in block $ (i(\frac{k+1}{2}-1)+(k-1))\pmod{k}$ and leaves in block $((i+1)(\frac{k+1}{2}-1)+(k-1)) \pmod{k}$. Element $i_l\in V$ is introduced in block $i ( \frac{k+1}{2})\pmod{k}$ and leaves in block $(i+1)( \frac{k+1}{2})\pmod{k}$. For any $0\leq j \leq k-1$, the number of introductions and leavings in $B_j$ is exactly two.  Please see Examples~\ref{Example: constructing minimum DCCD(8,5,5)} and \ref{Example: constructing minimum DCCD(12,7,7)} for two designs constructed as such.
		
		Now we show all pairs are covered. Between blocks $B_0,B_{\frac{k-1}{2}},$ and $B_{k-1}$, the pairs $\{i,\infty_s\}$ and $\{i,\infty_l\}$ for $i\in V\backslash \{\infty_s,\infty_l\}$ and $\{\infty_s,\infty_l\}$ are covered.  All elements $i,m \in V \backslash \{\infty_s, \infty_l\}$ are in at least $\frac{k+1}{2}$ consecutive blocks. As there are $k$ blocks, any two runs of at least $\frac{k+1}{2}$ consecutive blocks must intersect, so all pairs $\{i,m\}$, $i \neq m$, are covered. Thus, we have a valid minimum \dcccd($2k-2,k,k$).
	\end{proof}

	\begin{example}\label{Example: constructing minimum DCCD(8,5,5)}
		We construct a minimum \dcccd{}(8,5,5) in Table~\ref{Table: A minimum dcccd(8,5,5)} using Theorem~\ref{Theorem: minimum construction for dcccd(2k-2,k,k)}. 
		We highlight the long run elements by alternating \textcolor{\aPink}{pink} and \textcolor{\aBurgundy}{burgundy} and the short run elements by alternating \textcolor{\aLightGreen}{light green} and \textcolor{\aDarkGreen}{dark green}. This alternating emphasizes that when each short (long) run element leaves, the next short (long) run element is introduced in that same block. 
		\begin{table}[ht!]   
			\centering
			\begin{tabular}{lllll}
				$B_0$ & $B_1$ & $B_2$ & $B_3$ & $B_4$\\
				
				$\infty_{l*}^*$  & $\textcolor{\aDarkGreen}{1_s}^*$ & $\textcolor{\aDarkGreen}{1_s}$           & $\textcolor{\aDarkGreen}{1_s}_*$   & $\textcolor{\aLightGreen}{0_s}^*$      \\
				
				$\textcolor{\aPink}{0_l}^*$       & $\textcolor{\aPink}{0_l}$   & $\textcolor{\aPink}{0_l}$           & $\textcolor{\aPink}{0_l}_*$   & $\infty_{s*}^*$ \\
				
				$\textcolor{\aLightGreen}{0_s}$         & $\textcolor{\aLightGreen}{0_s}_*$   & $\infty_{l*}^*$    & $\textcolor{\aBurgundy}{1_l}^*$ & $\textcolor{\aBurgundy}{1_l}$      \\
				
				$\textcolor{\aBurgundy}{1_l}$         & $\textcolor{\aBurgundy}{1_l}_*$   & $\infty_{s*}^*$    & $\textcolor{\aLightGreen}{2_s}^*$ & $\textcolor{\aLightGreen}{2_s}$      \\
				
				$\textcolor{\aLightGreen}{2_s}_*$         & $\textcolor{\aPink}{2_l}^*$ & $\textcolor{\aPink}{2_l}$           & $\textcolor{\aPink}{2_l}$   & $\textcolor{\aPink}{2_l}_*$     
			\end{tabular}
			\caption{A minimum \dcccd{}(8,5,5)}
			\label{Table: A minimum dcccd(8,5,5)}
		\end{table}
	\end{example}
	
	\begin{example}\label{Example: constructing minimum DCCD(12,7,7)}
		We construct a minimum \dcccd{}(12,7,7) in Table~\ref{Table: A minimum dcccd(12,7,7)} using Theorem~\ref{Theorem: minimum construction for dcccd(2k-2,k,k)}. 
		We highlight the elements as we did in Example~\ref{Example: constructing minimum DCCD(8,5,5)}.
		\begin{table}[ht!]   
			\centering
			\begin{tabular}{lllllll}
				$B_0$ & $B_1$ & $B_2$ & $B_3$ & $B_4$ & $B_5$ & $B_6$ \\
				
				$\infty_{l*}^*$ & $\textcolor{\aDarkGreen}{3_s}^*$  & $\textcolor{\aDarkGreen}{3_s}$    & $\textcolor{\aDarkGreen}{3_s}$           & $\textcolor{\aDarkGreen}{3_s}_*$ & $\textcolor{\aLightGreen}{2_s}^*$  & $\textcolor{\aLightGreen}{2_s}$           \\
				
				$\textcolor{\aPink}{0_l}^*$         & $\textcolor{\aPink}{0_l}$    & $\textcolor{\aPink}{0_l}$    & $\textcolor{\aPink}{0_l}$           & $\textcolor{\aPink}{0_l}_*$ & $\textcolor{\aBurgundy}{3_l}^*$  & $\textcolor{\aBurgundy}{3_l}$           \\
				
				$\textcolor{\aLightGreen}{0_s}$           & $\textcolor{\aLightGreen}{0_s}$    & $\textcolor{\aLightGreen}{0_s}_*$ & $\infty_{s*}^*$ & $\textcolor{\aBurgundy}{1_l}^*$  & $\textcolor{\aBurgundy}{1_l}$    & $\textcolor{\aBurgundy}{1_l}$           \\
				
				$\textcolor{\aBurgundy}{1_l}$           & $\textcolor{\aBurgundy}{1_l}_*$ & $\textcolor{\aPink}{4_l}^*$  & $\textcolor{\aPink}{4_l}$           & $\textcolor{\aPink}{4_l}$    & $\textcolor{\aPink}{4_l}$    & $\textcolor{\aPink}{4_l}_*$        \\
				
				$\textcolor{\aLightGreen}{4_s}_*$        & $\textcolor{\aPink}{2_l}^*$  & $\textcolor{\aPink}{2_l}$    & $\textcolor{\aPink}{2_l}$           & $\textcolor{\aPink}{2_l}$    & $\textcolor{\aPink}{2_l}_*$ & $\textcolor{\aLightGreen}{0_s}^*$         \\
				
				$\textcolor{\aBurgundy}{3_l}$           & $\textcolor{\aBurgundy}{3_l}$    & $\textcolor{\aBurgundy}{3_l}_*$ & $\infty_{l*}^*$ & $\textcolor{\aLightGreen}{4_s}^*$  & $\textcolor{\aLightGreen}{4_s}$    & $\textcolor{\aLightGreen}{4_s}$           \\
				
				$\textcolor{\aLightGreen}{2_s}$           & $\textcolor{\aLightGreen}{2_s}_*$ & $\textcolor{\aDarkGreen}{1_s}^*$  & $\textcolor{\aDarkGreen}{1_s}$           & $\textcolor{\aDarkGreen}{1_s}$    & $\textcolor{\aDarkGreen}{1_s}_*$ & $\infty_{s*}^*$
				
			\end{tabular}
			\caption{A minimum \dcccd{}(12,7,7)}
			\label{Table: A minimum dcccd(12,7,7)}
		\end{table}
	\end{example}
	
	\begin{proposition} \label{Proposition: dcccd(2k-2,k,k-1) has an exansion set iff k = 4}
		A \tdcccd{}($2k-2,k,k-1)$ has an expansion set if and only if $k=4$. 
	\end{proposition}
	\begin{proof}
		We have $v=2k-2=2(k-2)+2$, so $(k-2)|v$ if and only if $k=3,4$. When $k=3$ the \dcccd{} cannot exist as $v<k+2$. When $k=4$ the \dcccd{}(6,4,3) has an expansion set as shown in Table~\ref{Table: tdcccd(6,4,2,3)}. All \dcccd{}(6,4,3) are isomorphic so all have an expansion set.
	\end{proof}
	
	From this \dcccd{}(6,4,3), using Proposition~\ref{Proposition: dcccd(2k-2,k,k-1) has an exansion set iff k = 4} in conjunction with the 1-factorization of $K_4$ on vertices $\{a,b,c,d\}$ and Theorem~\ref{Theorem: dccd(v+1,k,2) increase} we build the \tdcccd{}(10,4,9) seen in Table~\ref{Table: tdcccd(10,4,2,9)}.
	
	\begin{table}[ht!]
		\centering
		\begin{tabular}{lllllllll}
			$B_1$ & $B_2$ & $B_3$ & $B_4$ & $B_5$ & $B_6$ & $B_7$ & $B_8$ & $B_9$  \\
			0 & a$^*$ & c$^*$ & 4$^*$ & 4 & 4 & 4 & a$^*$ & b$^*$ \\
			1 & b$^*$ & d$^*$ & 5$^*$ & 5 & 5 & 5 & d$^*$ & c$^*$ \\
			2$^*$ & 2 & 2 & 2 & a$^*$ & b$^*$ & 0$^*$ & 0 & 0 \\
			3$^*$ & 3 & 3 & 3 & c$^*$ & d$^*$ & 1$^*$ & 1 & 1
		\end{tabular}
		\caption{A \tdcccd{}(10,4,9)}
		\label{tdcccd(10,4,2,9)}
		\label{Table: tdcccd(10,4,2,9)}
	\end{table}

	As every element in a \tdcccd{}$(2k-2,k,k-1)$ is introduced exactly once, adding a single new element to every block in the \dcccd{} constructed in Theorem~\ref{Theorem: tdcccd(4k-2,2k,2k-1) exist} produces a design where every pair is covered and no pairs are covered more than once.
	
	\begin{theorem}
		A \tdcccd{}$(2k-3,k,k-2)$ exists for all odd $k\geq 3$.
	\end{theorem}
	
	
	Beyond manipulating a \scccd{} by doubling elements, we can directly construct \dcccd{} using difference like methods. Recall that $\mathbb{Z}_v =\{0,1,\hdots,v-1\}$ and every pair of distinct elements has difference $x-y = \pm d$ for some $1 \leq d \leq (v-1)/2$.  Example~\ref{Example: dcccd(11,4,11), c=1} will help us gain some intuition for constructing \dcccd{} with difference methods.

	
	\normalsize
	
	\begin{example}\label{Example: dcccd(11,4,11), c=1}
		We construct a \tdcccd{}(11,4,11). The element set will be the cyclic group $\mathbb{Z}_{11}$.  Figure~\ref{Figure: K11} shows a single base block 
		\{\textcolor{\aLightGreen}{0},1,2,\textcolor{\aLightGreen}{5}\} and its successor \{{1},2,\textcolor{\aPink}{3},\textcolor{\aPink}{6}\} after incrementing the base block by $1$. The vertices only in the first block are highlighted in \textcolor{\aLightGreen}{light green}, the vertices only on the second block are highlighted in \textcolor{\aPink}{pink}, and the vertices in the unchanged subset are coloured black. We connect each introduced element and the other elements in the block with a coloured edge, \textcolor{\aLightGreen}{light green} for the first block and \textcolor{\aPink}{pink} for the second; these are the pairs covered in each block and every difference in $\mathbb{Z}_{11}$ appears as an edge length in each block. Each subsequent block is obtained by incrementing (rotating) the previous block by one. After developing the blocks fully in $\mathbb{Z}_{11}$ all pairs will be covered exactly once. We also see pictorially that there will be double-changes between consecutive blocks. Table~\ref{Table: dcccd(11,4,2,11)} shows all the blocks of the thus constructed \dcccd{}(11,4,11) with the same colour coding on these two blocks.

		\begin{figure}[ht]
			\centering
			\begin{tikzpicture}
	\pgfmathparse{90 + (0 * (360 /11))}
	\node[{minimum size=1mm, circle,draw}] (0) at (\pgfmathresult:\radForExample) {0};
	\pgfmathparse{90 + (10 * (360 /11))}
	\node[{minimum size=1mm, circle,draw}] (1) at (\pgfmathresult:\radForExample) {1};
	\pgfmathparse{90 + (9 * (360 /11))}
	\node[{minimum size=1mm, circle,draw}] (2) at (\pgfmathresult:\radForExample) {2};
	\pgfmathparse{90 + (8 * (360 /11))}
	\node[{minimum size=1mm, circle,draw}] (3) at (\pgfmathresult:\radForExample) {3};
	\pgfmathparse{90 + (7 * (360 /11))}
	\node[{minimum size=1mm, circle,draw}] (4) at (\pgfmathresult:\radForExample) {4};
	\pgfmathparse{90 + (6 * (360 /11))}
	\node[{minimum size=1mm, circle,draw}] (5) at (\pgfmathresult:\radForExample) {5};
	\pgfmathparse{90 + (5 * (360 /11))}
	\node[{minimum size=1mm, circle,draw}] (6) at (\pgfmathresult:\radForExample) {6};
	\pgfmathparse{90 + (4 * (360 /11))}
	\node[{minimum size=1mm, circle,draw}] (7) at (\pgfmathresult:\radForExample) {7};
	\pgfmathparse{90 + (3 * (360 /11))}
	\node[{minimum size=1mm, circle,draw}] (8) at (\pgfmathresult:\radForExample) {8};
	\pgfmathparse{90 + (2 * (360 /11))}
	\node[{minimum size=1mm, circle,draw}] (9) at (\pgfmathresult:\radForExample) {9};
	\pgfmathparse{90 + (1 * (360 /11))}
	\node[{minimum size=1mm, circle,draw}] (10) at (\pgfmathresult:\radForExample) {\tiny{10}};
	\draw [-,cap=rect,line width=0.3mm,draw=\aLightGreen] (0)--(2);
	\draw [-,cap=rect,line width=0.3mm,draw=\aLightGreen] (0)--(5);
	\draw [-,cap=rect,line width=0.3mm,draw=\aLightGreen] (1)--(2);
	\draw [-,cap=rect,line width=0.3mm,draw=\aLightGreen] (1)--(5);
	\draw [-,cap=rect,line width=0.3mm,draw=\aLightGreen] (2)--(5);
	
	\draw [-,cap=rect,line width=0.3mm,draw=\aPink] (1)--(3);
	\draw [-,cap=rect,line width=0.3mm,draw=\aPink] (1)--(6);
	\draw [-,cap=rect,line width=0.3mm,draw=\aPink] (2)--(3);
	\draw [-,cap=rect,line width=0.3mm,draw=\aPink] (2)--(6);
	\draw [-,cap=rect,line width=0.3mm,draw=\aPink] (3)--(6);

	
	{
		\pgfmathparse{90 + (0 * (360 /11))}
		\node[{minimum size=1mm, circle, draw, fill=black, text=white}] (0) at (\pgfmathresult:\radForExample) {0};
		\pgfmathparse{90 + (10 * (360 /11))}
		\node[{minimum size=1mm, circle, draw, fill=black, text=white}] (1) at (\pgfmathresult:\radForExample) {1};
		\pgfmathparse{90 + (9 * (360 /11))}
		\node[{minimum size=1mm, circle, draw, fill=\aLightGreen, text=white}] (2) at (\pgfmathresult:\radForExample) {2};
		\pgfmathparse{90 + (6 * (360 /11))}
		\node[{minimum size=1mm, circle, draw, fill=\aLightGreen, text=white}] (5) at (\pgfmathresult:\radForExample) {5};
	}
	{
		\pgfmathparse{90 + (0 * (360 /11))}
		\node[{minimum size=1mm, circle, draw, fill=\aLightGreen, text=white}] (0) at (\pgfmathresult:\radForExample) {0};
		\pgfmathparse{90 + (6 * (360 /11))}
		\node[{minimum size=1mm, circle, draw, fill=\aLightGreen, text=white}] (5) at (\pgfmathresult:\radForExample) {5};
		\pgfmathparse{90 + (10 * (360 /11))}
		\node[{minimum size=1mm, circle, draw, fill=black, text=white}] (1) at (\pgfmathresult:\radForExample) {1};
		\pgfmathparse{90 + (9 * (360 /11))}
		\node[{minimum size=1mm, circle, draw, fill=black, text=white}] (2) at (\pgfmathresult:\radForExample) {2};
		\pgfmathparse{90 + (8 * (360 /11))}
		\node[{minimum size=1mm, circle, draw, fill=\aPink, text=white}] (3) at (\pgfmathresult:\radForExample) {3};
		\pgfmathparse{90 + (5 * (360 /11))}
		\node[{minimum size=1mm, circle, draw, fill=\aPink, text=white}] (6) at (\pgfmathresult:\radForExample) {6};
	}
	
\end{tikzpicture}
			\caption{Visualizing the cyclic difference construction for \dccd(11,4,11)}
			\label{Figure: K11}
		\end{figure}
		
	\end{example}
	
	\begin{table}[ht!]
		\centering
		\begin{tabular}{lllllllllll}
			{\textcolor{\aLightGreen}{$B_{0,0}$}} & 
			{\textcolor{\aPink}{$B_{0,1}$}} & 
			{$B_{0,2}$} & {$B_{0,3}$} & {$B_{0,4}$} & 
			{$B_{0,5}$} & {$B_{0,6}$} & {$B_{0,7}$} & 
			{$B_{0,8}$} & {$B_{0,9}$} & {$B_{0,10}$} \\
			
			\textcolor{\aLightGreen}{0} & \textcolor{\aPink}{3}$^*$ & 3 & 3 & 6$^*$ & 6  & 6 & 9$^*$ & 9  & 9  & 1$^*$  \\
			\textcolor{black}{1} & \textcolor{black}{1} & 4$^*$ & 4 & 4 & 7$^*$  & 7 & 7 & 10$^*$ & 10 & 10 \\
			\textcolor{black}{2}$^*$ & \textcolor{black}{2} & 2 & 5$^*$ & 5 & 5  & 8$^*$ & 8 & 8  & 0$^*$  & 0  \\
			\textcolor{\aLightGreen}{5}$^*$ & \textcolor{\aPink}{6}$^*$ & 7$^*$ & 8$^*$ & 9$^*$ & 10$^*$ & 0$^*$ & 1$^*$ & 2$^*$  & 3$^*$  & 4$^*$ 
		\end{tabular}
		\caption{A \tdcccd{}(11,4,11) }
		\label{Table: DCCCD(11,4,2,11)}\label{Table: dcccd(11,4,2,11)}\label{Table: dcccd(11,4,11)}
	\end{table}

	

	
	In Example~\ref{Example: dcccd(11,4,11), c=1} the construction used a single base block which is developed incrementally in a cyclic group. We can construct more cyclic \tdcccd{} in a similar way using more than one base block. 
	\begin{theorem}[cyclic construction of \tdcccd{} with $c$ base blocks]
		\label{Theorem: Initial c blocks for all k >= 3}
		
		A \tdcccd{}$(c(4k-6)+1,k,c^2(4k-6)+c)$ exists for all $k\geq 3$ when $1\leq c \leq 5$.
	\end{theorem}
	
	\begin{proof}
		Let $v = c(4k-6)+1$ and $A= \{0,1,...,k-3\} \subset \mathbb{Z}_v$. Let $B_{i,0} \subseteq \mathbb{Z}_v$ for $0 \leq i \leq c-1$ be the blocks listed in Table~\ref{Table: initial blocks for dcccd(c(4k-6)+1,k,c^2(4k-6)+c)} for the corresponding $c$. Define $B_{i,j} = B_{i,0}+j$ with addition in $\mathbb{Z}_v$.  For each $j \in \mathbb{Z}_v$, let $L_j = (B_{0,j}, B_{1,j}, \hdots, B_{c-1,j})$. Let $\cL $ be the concatenation of the $L_j$ in order $0 \leq j \leq v-1$. $\cL$ is a circular double-change list of blocks where the elements introduced in block $B_{i,j}$ are $B_{i,j}\backslash (A+j)$. For each $1 \leq c \leq 5$ the pairs covered in the blocks of $L_0$ (and indeed in any $L_j$) contain each difference in $\mathbb{Z}_v$ exactly once.  Since the $L_j$ are simply the development of $L_0$ over the entire group, every pair will be covered exactly once.  We give details of this for $c=1,2$.
		
		Case $c=1$: Since the elements introduced in every $B_{i,0}$ are  $B_{i,0}\backslash A$, the pairs covered are the pair of new elements and all the pairs containing one new element and one element of $A$.  Calculating all the differences from these pairs, block $B_{0,0}$ covers the differences $\pm(k-1)$, $\pm\{1,2,\hdots,k-2\}$ and $\pm\{k,k+1,\hdots,2k-3\}$. This is all the differences in $\mathbb{Z}_v = \mathbb{Z}_{4k-5}$.
		
		Case $c=2$: Similarly to $c=1$, the pairs covered in $B_{0,0}$ and $B_{1,0}$ are the pair of new elements and the pairs containing one new element and one element of $A$.  The differences from these pairs in $B_{0,0}$ are ,$\pm(3k-4)$, $\pm\{k-1,k,\hdots,2k-4\}$ and $\pm\{3k-3,3k-2,\hdots,4k-6 \}$.  Block $B_{1,0}$ covers the differences $\pm(2k-3)$, $\pm \{1,2,\hdots, k-2\}$ and $\pm\{2k-2, 2k-1,\hdots, 3k-5\}$. These differences are all distinct and their union is all the differences in $\mathbb{Z}_v=\mathbb{Z}_{8k-11}$.
		
		Case $c=3,4$ and $5$ are all similar. In each case the number of blocks, $b=cv$ matches $g_2(v,k)$ which is also an integer so the designs are tight.
		
		\begin{table}[ht!]
			\centering
			\begin{tabular}{|l|l|}
				\hline
				\textbf{$c$} & \textbf{Initial Blocks} \\ \hline
				
				1 & $B_{0,0} = A \cup \{k-2,2(k-2)+1\}$         \\ \hline
				
				2 & $B_{0,0} = A \cup \{2(k-2),5(k-2)+2\},$ \\ & $
				B_{1,0} = A \cup \{k-2,3(k-2)+1\}$               \\ \hline
				
				3 & $B_{0,0} = A \cup \{2(k-2),7(k-2)+3\}, $ \\ & $ 
				B_{1,0} = A \cup \{3(k-2)+1,5(k-2)+2\},$ \\ & $
				B_{2,0} = A \cup \{k-2,4(k-2)+2\}$              \\ \hline
				
				4 & $ B_{0,0} = A \cup \{2(k-2),8(k-2)+4\}, $ \\ & $
				B_{1,0} = A \cup \{3(k-2)+1,5(k-2)+2\},$ \\ & $
				B_{2,0} = A \cup \{4(k-2)+2,7(k-2)+4\}, $ \\ & $
				B_{3,0} = A \cup \{k-2,6(k-2)+3\}$    \\ \hline
				
				5 & $ B_{0,0} = A \cup \{2(k-2),6(k-2)+3\},$ \\ & $
				B_{1,0} = A \cup \{3(k-2)  , 9(k-2)+4\},$ \\ & $ 
				B_{2,0} = A \cup \{4(k-2)+2, 7(k-2)+4\},$ \\ & $
				B_{3,0} = A \cup \{5(k-2)+3, 8(k-2)+4\},$ \\ & $
				B_{4,0} = A \cup \{ k-2  ,10(k-2)+5\}$      \\ \hline
				
			\end{tabular}
			\caption{Base blocks for \tdcccd{}}
			\label{Table: initial blocks for dcccd(c(4k-6)+1,k,c^2(4k-6)+c)}
		\end{table}
		
	\end{proof}
	
	Example~\ref{Example: dcccd(11,4,11), c=1} is an instance of Theorem~\ref{Theorem: Initial c blocks for all k >= 3} when $c=1$. Example~\ref{Example: dcccd(21,4,42) when c=2} demonstrates Theorem~\ref{Theorem: Initial c blocks for all k >= 3} when $c=2$ and $k=4$.
	
	\begin{example}\label{Example: dcccd(21,4,42) when c=2}
		When $k=4$ and $c=2$ we construct the \tdcccd{}(21,4,42) in Table~\ref{Table: tdcccd(21,4,2,42) highlighted}. In this table we highlight every $\textcolor{\aLightGreen}{x}\in X$ such that the pair $(0,\textcolor{\aLightGreen}{x})$ is covered. This demonstrates that every difference is covered and since the blocks are developed cyclically in $\mathbb{Z}_{21}$, this guarantees that all pairs are covered exactly once.  
		
		\begin{table}[tbh!]
			\centering
			\setlength{\tabcolsep}{2.5pt}
			\begin{tabular}{lllllllllllllllllllll}
				\small{$B_{0,0}$} & \small{$B_{1,0}$} & \small{$B_{0,1}$} & \small{$B_{1,1}$} & \small{$B_{0,2}$} & \small{$B_{1,2}$} & \small{$B_{0,3}$} & \small{$B_{1,3}$} & \small{$B_{0,4}$} & \small{$B_{1,4}$} & \small{$B_{0,5}$} & \small{$B_{1,5}$} & \small{$B_{0,6}$} & \small{$B_{1,6}$} \\ 
				
				{0}      & {0}      & {5$^*$}  & {3$^*$}  & {3}      & {3}      & {3}      & {3}      & {8$^*$}  & {6$^*$}  & {6}      & {6}      & {6}      & {6}      \\
				{1}      & {1}      & {1}      & {1}      & {6$^*$}  & {4$^*$}  & {4}      & {4}      & {4}      & {4}      & {9$^*$}  & {7$^*$}  & {7}      & {7}      \\
				\textcolor{\aLightGreen}{4$^*$}  & \textcolor{\aLightGreen}{2$^*$}  & {2}      & {2}      & {2}      & {2}      & {7$^*$}  & {5$^*$}  & {5}      & {5}      & {5}      & {5}      & {10$^*$} & {8$^*$}  \\
				\textcolor{\aLightGreen}{12$^*$} & \textcolor{\aLightGreen}{7$^*$}  & {13$^*$} & {8$^*$}  & {14$^*$} & {9$^*$}  & {15$^*$} & {10$^*$} & {16$^*$} & {11$^*$} & {17$^*$} & {12$^*$} & {18$^*$} & {13$^*$} \\
				\hline 
				
				\small{$B_{0,7}$} & \small{$B_{1,7}$} & \small{$B_{0,8}$} & \small{$B_{1,8}$} & \small{$B_{0,9}$} & \small{$B_{1,9}$} & \small{$B_{0,10}$} & \small{$B_{1,10}$} & \small{$B_{0,11}$} & \small{$B_{1,11}$} & \small{$B_{0,12}$} & \small{$B_{1,12}$} & \small{$B_{0,13}$} & \small{$B_{1,13}$} \\
				{11$^*$} & {9$^*$}  & {9}      & {9}      & \textcolor{\aLightGreen}{9}      & {9}      & {14$^*$} & {12$^*$} & {12}     & {12}     & {12}     & {12}     & {17$^*$} & {15$^*$} \\
				
				{7}      & {7}      & {12$^*$} & {10$^*$} & \textcolor{\aLightGreen}{10}     & {10}     & {10}     & {10}     & {15$^*$} & {13$^*$} & {13}     & {13}     & {13}     & {13}     \\
				
				{8}      & {8}      & {8}      & {8}      & \textcolor{\aLightGreen}{13$^*$} & {11$^*$} & {11}     & {11}     & {11}     & {11}     & {16$^*$} & {14$^*$} & {14}     & {14}     \\
				
				{19$^*$} & {14$^*$} & {20$^*$} & {15$^*$} & {0$^*$}  & {16$^*$} & {1$^*$}  & {17$^*$} & {2$^*$}  & {18$^*$} & {3$^*$}  & {19$^*$} & {4$^*$}  & {20$^*$} \\
				\hline 
				
				\small{$B_{0,14}$} & \small{$B_{1,14}$} & \small{$B_{0,15}$} & \small{$B_{1,15}$} & \small{$B_{0,16}$} & \small{$B_{1,16}$} & \small{$B_{0,17}$} & \small{$B_{1,17}$} & \small{$B_{0,18}$} & \small{$B_{1,18}$} & \small{$B_{0,19}$} & \small{$B_{1,19}$} & \small{$B_{0,20}$} & \small{$B_{1,20}$} \\
				
				{15}     & \textcolor{\aLightGreen}{15}     & {15}     & {15}     & {20$^*$} & {18$^*$} & \textcolor{\aLightGreen}{18}     & {18}     & {18}     & {18}     & {2$^*$}  & {0$^*$}  & {0}      & {0}      \\
				{18$^*$} & \textcolor{\aLightGreen}{16$^*$} & {16}     & {16}     & {16}     & {16}     & {0$^*$}  & {19$^*$} & {19}     & {19}     & {19}     & \textcolor{\aLightGreen}{19}     & \textcolor{\aLightGreen}{3$^*$}  & \textcolor{\aLightGreen}{1$^*$}  \\
				{14}     & \textcolor{\aLightGreen}{14}     & {19$^*$} & {17$^*$} & {17}     & {17}     & \textcolor{\aLightGreen}{17}     & {17}     & {1$^*$}  & {20$^*$} & {20}     & \textcolor{\aLightGreen}{20}     & {20}     & {20}     \\
				{5$^*$}  & {0$^*$}  & {6$^*$}  & {1$^*$}  & {7$^*$}  & {2$^*$}  & \textcolor{\aLightGreen}{8$^*$}  & {3$^*$}  & {9$^*$}  & {4$^*$}  & {10$^*$} & \textcolor{\aLightGreen}{5$^*$}  & \textcolor{\aLightGreen}{11$^*$} & \textcolor{\aLightGreen}{6$^*$} 
			\end{tabular}
			\caption{A \tdcccd{}(21,4,42)}
			\label{tdcccd(21,4,2,42)}
			\label{Table: tdcccd(21,4,2,42)}
			\label{tdcccd(21,4,2,42) highlighted}
			\label{Table: tdcccd(21,4,2,42) highlighted}
			\setlength{\tabcolsep}{6pt}
		\end{table}
	\end{example}
	
	\begin{example} \label{Example: dcccd(13,3,2,26) pictorially}
		We may also visualize multiple base blocks pictorially as we did in Example~\ref{Example: dcccd(11,4,11), c=1}. Figure~\ref{Step 1} shows the base blocks \textcolor{\aLightGreen}{$B_{0,0}=\{\textcolor{black}{0},1,4\}$} and \textcolor{\aDarkGreen}{$B_{1,0}=\{\textcolor{black}{0},2,7\}$} from Theorem~\ref{Theorem: Initial c blocks for all k >= 3} for $c=2$ and $k=3$.  These construct a \tdcccd{}(13,3,26). These blocks are highlighted \textcolor{\aLightGreen}{light green} and \textcolor{\aDarkGreen}{dark green} respectively in Figure~\ref{Step 1} and the unchanged subset is highlighted in black. All differences from $1,2,...,12 \pmod{13}$ are covered in these two blocks. Figure~\ref{Step 2} shows how we increment block \textcolor{\aLightGreen}{$B_{0,0}$} in  Figure~\ref{Step 1} by one to get block \textcolor{\aPink}{$B_{0,1}$}.  The completed table of all blocks of this design in shown in Table~\ref{Table: tdcccd(13,3,26)}. The block \textcolor{\aBurgundy}{$B_{1,1}$} is highlighted in \textcolor{\aBurgundy}{burgundy} to emphasize this is the block progression from \textcolor{\aDarkGreen}{$B_{1,0}$}.

		\begin{figure}[htb]
			\begin{minipage}[t]{.5\textwidth}
				\centering
				\scalebox{0.8}{
					\begin{tikzpicture}
	\foreach \x [count=\idx from 0] in {12,...,0} {
		\pgfmathparse{90 + ((\idx+1) * (360 /13))}
		\node at (\pgfmathresult:\radForExampleThree cm) [circle,draw=black,fill=white, minimum size=1mm,line width=0.2mm] {\x}; 
	};

	\pgfmathparse{90 + (0 * (360 /13))}
	\node[{minimum size=1mm, circle, draw, fill=black, text=white}] (0) at (\pgfmathresult:\radForExampleThree) {0};
	\pgfmathparse{90 + (12 * (360 /13))}
	\node[{minimum size=1mm, circle, draw, fill=\aLightGreen, text=white}] (1) at (\pgfmathresult:\radForExampleThree) {1};
	\pgfmathparse{90 + (9 * (360 /13))}
	\node[{minimum size=1mm, circle, draw, fill=\aLightGreen, text=white}] (4) at (\pgfmathresult:\radForExampleThree) {4};
	\pgfmathparse{90 + (11 * (360 /13))}
	\node[{minimum size=1mm, circle, draw, fill=\aDarkGreen, text=white}] (2) at (\pgfmathresult:\radForExampleThree) {2};
	\pgfmathparse{90 + (6 * (360 /13))}
	\node[{minimum size=1mm, circle, draw, fill=\aDarkGreen, text=white}] (7) at (\pgfmathresult:\radForExampleThree) {7};
	
	\begin{pgfonlayer}{background}
		\draw [-,cap=rect,line width=0.3mm,draw=\aLightGreen] (0)--(1);
		\draw [-,cap=rect,line width=0.3mm,draw=\aLightGreen] (0)--(4);
		\draw [-,cap=rect,line width=0.3mm,draw=\aLightGreen] (1)--(4);
		
		\draw [-,cap=rect,line width=0.3mm,draw=\aDarkGreen] (0)--(2);
		\draw [-,cap=rect,line width=0.3mm,draw=\aDarkGreen] (0)--(7);
		\draw [-,cap=rect,line width=0.3mm,draw=\aDarkGreen] (2)--(7);
	\end{pgfonlayer}

\end{tikzpicture}
				}
				\subcaption{\textcolor{\aLightGreen}{$B_{0,0}$} and \textcolor{\aDarkGreen}{$B_{1,0}$} }
				\label{Step 1}
			\end{minipage}
			\hfill
			\begin{minipage}[t]{.5\textwidth}
				\centering
				\scalebox{0.8}{
					\begin{tikzpicture}
	\foreach \x [count=\idx from 0] in {12,...,0} {
		\pgfmathparse{90 + ((\idx+1) * (360 /13))}
		\node at (\pgfmathresult:\radForExampleThree cm) [circle,draw=black,fill=white, minimum size=1mm,line width=0.2mm] {\x}; 
	};
	
	\pgfmathparse{90 + (0 * (360 /13))}
	\node[{minimum size=1mm, circle, draw, fill=\aDarkGreen, text=white}] (0) at (\pgfmathresult:\radForExampleThree) {0};
	\pgfmathparse{90 + (12 * (360 /13))}
	\node[{minimum size=1mm, circle, draw, fill=\aPink, text=white}] (1) at (\pgfmathresult:\radForExampleThree) {1};
	\pgfmathparse{90 + (8 * (360 /13))}
	\node[{minimum size=1mm, circle, draw, fill=\aPink, text=white}] (5) at (\pgfmathresult:\radForExampleThree) {5};
	\pgfmathparse{90 + (11 * (360 /13))}
	\node[{minimum size=1mm, circle, draw, fill=black, text=white}] (2) at (\pgfmathresult:\radForExampleThree) {2};
	\pgfmathparse{90 + (6 * (360 /13))}
	\node[{minimum size=1mm, circle, draw, fill=\aDarkGreen, text=white}] (7) at (\pgfmathresult:\radForExampleThree) {7};
	
	\begin{pgfonlayer}{background}
		\draw [-,cap=rect,line width=0.3mm,draw=\aPink] (1)--(2);
		\draw [-,cap=rect,line width=0.3mm,draw=\aPink] (1)--(5);
		\draw [-,cap=rect,line width=0.3mm,draw=\aPink] (2)--(5);
		
		\draw [-,cap=rect,line width=0.3mm,draw=\aDarkGreen] (0)--(2);
		\draw [-,cap=rect,line width=0.3mm,draw=\aDarkGreen] (0)--(7);
		\draw [-,cap=rect,line width=0.3mm,draw=\aDarkGreen] (2)--(7);
	\end{pgfonlayer}

	
\end{tikzpicture}
				}
				\subcaption{\textcolor{\aDarkGreen}{$B_{1,0}$} and \textcolor{\aPink}{ $B_{0,1}$}}
				\label{Step 2}
			\end{minipage}
			\caption{Step 1 (\ref{Step 1}) and Step 2 (\ref{Step 2}) of the cyclic construction of a \tdcccd{}(13,3,26).} 
			\label{Figure: Steps in DCCD(13,3,26)}
		\end{figure}

		
			

		\begin{table}[ht!]
			\centering
			\setlength{\tabcolsep}{2pt}
			\begin{tabular}{lllllllllllllllll}
				
				\textcolor{\aLightGreen}{{$B_{0,0}$}} & {\textcolor{\aDarkGreen}{$B_{1,0}$}} & 
				{\textcolor{\aPink}{$B_{0,1}$}} & {\textcolor{\aBurgundy}{$B_{1,1}$}} & 
				{$B_{0,2}$} & {$B_{1,2}$} & 
				{$B_{0,3}$} & {$B_{1,3}$} & 
				{$B_{0,4}$} & {$B_{1,4}$} & 
				{$B_{0,5}$} & {$B_{1,5}$} & 
				{$B_{0,6}$}    \\
				\textcolor{\aLightGreen}{0}$^*$ & {0} & \textcolor{\aPink}{1}$^*$ & \textcolor{\aBurgundy}{1} & 2$^*$ & 2 & 3$^*$ & 3 & 4$^*$ & 4 & 5$^*$ & 5 & 6$^*$  \\
				\textcolor{\aLightGreen}{1} & \textcolor{\aDarkGreen}{2}$^*$ & {2} & \textcolor{\aBurgundy}{3}$^*$ & 3 & 4$^*$ & 4 & 5$^*$ & 5 & 6$^*$ & 6 & 7$^*$ & 7    \\
				\textcolor{\aLightGreen}{4}$^*$ & \textcolor{\aDarkGreen}{7}$^*$ & \textcolor{\aPink}{5}$^*$ & \textcolor{\aBurgundy}{8}$^*$ & 6$^*$ & 9$^*$ & 7$^*$ & 10$^*$ & 8$^*$ & 11$^*$ & 9$^*$ & 12$^*$ & 10$^*$  \\ \hline 
				{$B_{1,6}$} & {$B_{0,7}$} &
				{$B_{1,7}$} & {$B_{0,8}$} & 
				{$B_{1,8}$} & {$B_{0,9}$} & 
				{$B_{1,9}$} & {$B_{0,10}$}  & 
				{$B_{1,10}$} & {$B_{0,11}$} & 
				{$B_{1,11}$} & {$B_{0,12}$} & 
				{$B_{1,12}$} \\
				6 & 7$^*$ & 7 & 8$^*$ & 8 & 9$^*$ & 9 & 10$^*$ & 10$^*$ & 11$^*$ & 11 & 12$^*$ & 12 \\
				8$^*$ & 8 & 9$^*$ & 9 & 10$^*$ & 10 & 11$^*$ & 11 & 12$^*$ & 12 & 0$^*$ & 0 & 1$^*$ \\
				0$^*$ & 11$^*$ & 1$^*$ & 12$^*$ & 2$^*$ & 0 & 3 & 1 & 4 & 2 & 5 & 3 & 6 
			\end{tabular}
			\caption{Tight \dcccd{}(13,3,26)}
			\label{Table: tdcccd(13,3,2,26)}\label{Table: tdcccd(13,3,26)}
			\setlength{\tabcolsep}{6pt}
		\end{table}
	\end{example}
	
	We believe that this method will work to construct \tdcccd{} for larger $c$ but finding base blocks that work is more of a challenge as $c$ grows. 
	
	The \tdcccd{} constructed in Theorem~\ref{Theorem: Initial c blocks for all k >= 3} have expansion sets when divisibility conditions are met.
	\begin{proposition}
		When $(k-2) | (2c+1) $, a \tdcccd{} constructed via Theorem~\ref{Theorem: Initial c blocks for all k >= 3} has an expansion set.
	\end{proposition}
	
	\begin{proof}
		In Theorem~\ref{Theorem: Initial c blocks for all k >= 3}, the unchanged subsets are all $A+j \subset \mathbb{Z}_v$, $j\in \mathbb{Z}_v$. When $(k-2)|(2c+1)$ then $\frac{v}{k-2}=4c+\frac{2c+1}{k-2}$. Taking $0 \leq j < \frac{v}{k-2}-1$, the unchanged subsets $U_{ic(k-2)}$ partition $\mathbb{Z}_v$ giving an
		expansion set.
	\end{proof}
	
	Thus, in these cases, we can construct more \tdcccd{}.
	\begin{corollary} \label{Corollary: Odd number of expansion sets}
		If $l = \frac{2c+1}{k-2}$ is odd, then a \tdcccd{}$( 4ck -2c + l +2,k, 4kc^2 +2c^2 + 4cl+3c  + (l^2+l)/2 )$ exists.
	\end{corollary}
	\begin{proof}
		The order of the expansion set is $|\cE|=\frac{4ck-6c+1}{k-2}=4c+\frac{2c+1}{k-2}$. This is odd if and only if $\frac{2c+1}{k-2}$ is odd. Then we apply Theorem~\ref{Theorem: dccd(v+1,k,2) increase}. 
	\end{proof}
	
	Using the \tdcccd{} constructed via Theorem~\ref{Theorem: Initial c blocks for all k >= 3} and Corollary~\ref{Corollary: Odd number of expansion sets}, we obtain more \tdcccd{}.
	\begin{corollary}
		\label{Corollary: tdcccd(15,3,2,35) exists}
		\label{Corollary: tdcccd(15,3,35) exists}
		\label{Corollary: tdcccd(21,5,30) exists}
		\label{Corollary: 11 additional tdcccd using recusion on diff families}
		
		The following eleven \tdcccd{} exist: \dccd{}(15,3,35), \dccd{}(21,5,30), \dccd{}(27,3,117), \dccd{}(55,7,135), \dccd{}(39,3,247), \dccd{}(105,9,364), \dccd{}(67,3,489), \dccd{}(77,5,418), \dccd{}(161,11,640), \dccd{}(63,3,651),  \dccd{}(253,13,1386)
	\end{corollary}

	
	We have found appropriate base blocks for one instance of the cyclic construction when $c=6$.
	
	\begin{proposition}
		\label{Proposition: tdcccd(61,4,366) exists}
		A \tdcccd{}$(61,4,366)$ exists. 
	\end{proposition}

	\begin{proof}
		The proof is identical to that of Theorem~\ref{Theorem: Initial c blocks for all k >= 3} except that $c=6$ and the block size is fixed at $k=4$.  The base blocks are given in Table~\ref{Table: c=6, k=4 dcccd(61,4,366) example} with $A = \{0,1\}$.
	\end{proof}

	\begin{table}[ht!]
		\centering
		\begin{tabular}{|l|l|l|}
			\hline
			\textbf{$k$}  & \textbf{$c$} & \textbf{Initial Blocks} 
			\\ \hline
			4  & 6 & $ B_{0,0}= A \cup \{4, 19\}, B_{1,0}=A \cup \{ 6, 22\}, B_{2,0}= A \cup \{8, 25\},$ \\
			&& $B_{3,0}= A \cup \{10, 48\}, B_{4,0}= A \cup \{12, 32\},B_{5,0}= A \cup \{2, 28\}$ \\ \hline
			
		\end{tabular}
		\caption{Base blocks for a \tdcccd{}(61,4,366)}
		\label{Table: c=6, k=4 dcccd(61,4,366) example}
	\end{table}

	\section{Conclusion}
	We have proven a recursion that uses 1-factorizations and expansion sets to construct \dccd{}($v+\frac{v}{k-2}+1,k,b+\frac{v}{k-2}\frac{v+k-2}{2k-4}$) from \dccd{}($v,k,b$) when $\frac{v}{k-2}$ is odd.  
	We use a family of \tscccd{}($2k'-1,k',2k'-1)$ to build \tdcccd{}($2k-2,k,k-1$) for $k$ even and \tdcccd{}($2k-3,k,k-2$) for $k$ odd.
	We prove for $k$ odd \dcccd{}$(2k-2,k,k)$ are minimum and construct them for all $k\geq 5$.
	We use difference like methods to construct \tdcccd{}(61,4,366) and \tdcccd{}($c(4k-6)+1,k,c^2(4k-6)+c$) for $k \geq 3$, $1 \leq c \leq 5$. When possible, we apply the recursion to these to construct eleven additional \tdcccd{}.
	
	There are several directions to explore next. Can we develop a version of Theorem~\ref{Theorem: dccd(v+1,k,2) increase} that works when the size of the expansion set is even? 
	Can we develop a version of Theorem~\ref{Theorem: dccd(v+1,k,2) increase} to build an economical linear (circular) \dccd{} from a tight linear (circular) \dccd{} as in \cite{Chafee_sccd}? 
	Does there exist a recursive construction that uses expansion sets and two \dccd{} to build tight \ldccd{} or \dcccd{} as in \cite{sccd,Chafee_sccd} respectively? 
	McSorley \cite{circular} determined the number of blocks in minimum \scccd{} as long as $v\leq 2k-1$ and constructed minimum \scccd{} for all these parameters.  Can we similarly solve the existence of minimum \dcccd{} when $v\leq 2k-2$? Further, can we find more linear (circular) \dccd{} for $k>4$? Alternatively, in this paper we only considered covering pairs, can we find linear (circular) \dccd{} where we cover triples or other larger tuples?

	\normalsize
	\bibliographystyle{plain}
	\bibliography{References}
\end{document}